\documentclass{amsart}
\usepackage{hyperref}
\usepackage{geometry}
\usepackage{enumerate}
\usepackage{amsmath,amssymb,amsthm,mathrsfs,amsfonts,dsfont}
\usepackage{graphicx}
\usepackage{tikz}
\usetikzlibrary{shapes}
\usetikzlibrary{plotmarks}
\usetikzlibrary{arrows}
\usetikzlibrary{positioning}
\usepackage{tkz-graph}
\theoremstyle{definition}
\newtheorem{defn}{Definition}[section]
\newtheorem{prop}[defn]{Proposition}
\newtheorem{thm}[defn]{Theorem}

\newtheorem{lem}[defn]{Lemma}

\newtheorem*{prob}{Problem}
\newtheorem{mthm}{Theorem}
\newcommand\myeq{\mathrel{\stackrel{\makebox[0pt]{\mbox{\normalfont\tiny def}}}{=}}}
\newcommand\myiff{\mathrel{\stackrel{\makebox[0pt]{\mbox{\normalfont\tiny def}}}{\iff}}}
\def\CA{\mathsf{CA}}
\def\GAM{\mathsf{GAM^{\not=}}}
\def\LL{\mathcal{L}}
\def\d{\mathsf{Dr}}
\def\ds{\mathsf{Drs}}

\def\a#1{\mathfrak{#1}}
\def\f#1#2{\mathfrak{Fr}_{#1}\ds_{#2}}

\def\sub{\mathrm{sub}}
\def\color{\mathrm{color}}
\def\tg{\mathrm{tag}}
\def\ev{\mathrm{ev}}

\begin{document}
\title{First order logic without equality on relativized semantics}
\author{Amitayu Banerjee}
\address{Amitayu Banerjee\\ Department of Logic, Institute of Philosophy\\ E\"otv\"os Lor\'and University\\ Budapest, Hungary}
\email{banerjee.amitayu@gmail.com}
\author{Mohamed Khaled}
\address{Mohamed Khaled\\ Alfr\'ed R\'enyi Institute of Mathematics\\ Hungarian Academy of Sciences\\ Budapest, Hungary}
\email{khaled.mohamed@renyi.mta.hu}
\date{}
\begin{abstract}
Let $\alpha\geq 2$ be any ordinal. We consider the class $\ds_{\alpha}$ of relativized diagonal free set algebras of dimension $\alpha$. With same technique, we prove several important results concerning this class. Among these results, we prove that almost all free algebras of $\ds_{\alpha}$ are atomless, and none of these free algebras contains zero-dimensional elements other than zero and top element. The class $\ds_{\alpha}$ corresponds to first order logic, without equality symbol, with $\alpha$-many variables and on relativized semantics. Hence, in this variation of first order logic, there is no finitely axiomatizable, complete and consistent theory.
\end{abstract}
\maketitle
\section{Introduction}
In the middle of the twentieth century, A. Tarski introduced and initiated the investigation of cylindric algebras and relation algebras. These algebras are Boolean algebras with extra additive, closure and complemented operators. The theories of these algebras are directly related to the development of some versions of quantifier logics, e.g.,, classical first order logic. These theories (and the theories of the related structures) have found interesting realizations and applications in mathematics, computer science, philosophy and logic, c.f., e.g.,, \cite{berghammer}, \cite{givant}, \cite{manca-salibra}, \cite{plotkin} and \cite{duntsch}.

An important notion in the theories of these algebras is the notion of representable algebras. These algebras can be conceived as expansions of Boolean set algebras whose elements are unary relations to algebras whose elements are relations of higher ranks. The question whether every abstract algebra is isomorphic to a representable algebra is the algebraic equivalent of the completeness theorem for the corresponding logic. Representable algebras represent the semantics of the corresponding logic, while abstract algebras correspond to its syntactical side.

One can find well motivated appropriate notions of representable structures by first locating them while giving up classical semantical prejudices. It is hard to give a precise mathematical underpinning to such intuitions. What really counts at the end is a completeness theorem stating a natural fit between chosen intuitive concrete-enough, but perhaps not excessively concrete, semantics and well behaved, hopefully recursive, axiomatization. G\"odel's completeness theorem ties just one choice of logical validity in standard set theoretic modeling.

The classical concrete algebras are cylindric set algebras defined by A. Tarski, these are algebras of sets of sequences in which the top element is a square of the form ${^{\alpha}U}$, where $U$ is a non-empty set and $\alpha$ is the dimension. Other concrete algebras can be the \emph{relativized} versions of cylindric set algebras. The top element of a relativized set algebra is arbitrary subset $V\subseteq {^{\alpha}U}$ with operations defined like cylindric set algebras, but relativized to $V$. From the modal perspective, such top elements are called \emph{guards} ‘‘guarding the semantics’’.

The notion of a relativized algebra has been introduced in algebraic logic by L. Henkin and I. N\'emeti. Relativization was proved extremely potent in obtaining positive results in both algebraic and modal logic, the slogan being \emph{relativization turns negative results positive}. For instance, I. N\'emeti proved, in a seminal result, that the universal theory of relativized cylindric set algebras is decidable. The corresponding multimodal logic exhibits nice modal behavior and is regarded as the base for proposing the so-called guarded fragments of first order logic by H. Andr\'eka, J. van Benthem and I. N\'emeti in \cite{andbennem}. 

The important connections between relativized cylindric set algebras and guarded fragments are discussed in \cite{andbennem} and \cite{benthem}. More liberal versions of guarded fragments are the so-called \emph{loosely guarded}, \emph{clique guarded} and \emph{packed fragments of first order logic}, see \cite{benthem} and \cite[Definitions 19.1, 19.2, 19.3, pp. 586-589]{HHbook}. Relativized algebras and their related logics attracted many logicians and were shown to have several desirable properties, especially concerning decidability and complexity issues. They are widely applied in various areas of computer science and linguistics (e.g.,, description logics, database theory, combining logics), see \cite{AMdNdR99}, \cite{BMP13}, \cite{GHKL14} and \cite{pratt}.

The structures of free cylindric algebras are quite rich since they are able to capture the whole of first order logic, in a sense. One of the first things to investigate about these free algebras is whether they are atomic or not, i.e.,, whether their Boolean reducts are atomic or not. By an atomic Boolean algebra, we mean an algebra for which below every non-zero element there is an atom, i.e.,, a minimal non-zero element. Atoms in these free algebras correspond to finitely axiomatizable complete and consistent theories, while the atomicity of these free algebras correspond to the failure of G\"odel's incompleteness theorem for the corresponding logics. For more details about this correspondence, see \cite{nem86}, \cite{myphd}, \cite{myigpl}, \cite{zalan} and \cite{andnem2013}.

For a class $\mathsf{K}$ of algebras, and a cardinal $\beta>0$, $\a{Fr}_{\beta}\mathsf{K}$ stands for the $\beta$-generated free $\mathsf{K}$ algebra. In particular, for any ordinal $\alpha$, the class of all cylindric algebras of dimension $\alpha$ is denoted by $\CA_{\alpha}$, thus $\a{Fr}_{\beta}\CA_{\alpha}$ denotes the $\beta$-generated free cylindric algebra of dimension $\alpha$. The following are known:

\begin{itemize}
\item If $\beta\geq \omega$, then $\a{Fr}_{\beta}\CA_{\alpha}$ is atomless (has no atoms). This result is due to D. Pigozzi \cite[2.5.13]{hmt1} and it can be generalized easily to any class of Boolean algebras with operators. So, from now on, let us assume that $\beta<\omega$.
\item If $\alpha<2$ then the free algebra $\a{Fr}_{\beta}\CA_{\alpha}$ is finite,
hence atomic \cite[2.5.3 (i)]{hmt1}. Moreover, the free algebra $\a{Fr}_{\beta}\CA_2$ is infinite but still atomic \cite[2.5.3(ii), 2.5.7(ii)]{hmt1}.
\item If $3\leq \alpha<\omega$, then $\a{Fr}_{\beta}\CA_{\alpha}$ has infinitely many atoms \cite[2.5.9]{hmt1}, and it was posed as an open question, cf \cite[Problem 4.14]{hmt2}, whether it is atomic or not.
\item In \cite{nem86}, it was shown that $\a{Fr}_{\beta}\CA_{\alpha}$ is not atomic for $3\leq \alpha<\omega$. This was proven by an involved metalogical machinery, namely, G\"odel's incompleteness Theorem. Then the problem of finding purely algebraic proof of this fact was raised in \cite[Problem 4.14]{hmt2}. Such a proof, for $\alpha\geq 4$, was found by I. N\'emeti \cite{nem84}. The problem of finding algebraic proof for the case $\alpha=3$ is still open.
\end{itemize}

Similar results concerning representable cylindric algebras are also obtained, c.f. \cite{nem84}. The question whether the finitely generated free relativized cylindric set algebras are atomic was a difficult problem that remained open for three decades. See \cite[Remark 18 (i)]{nem86}, \cite[Problem 38]{AMN91} and \cite[Problem 1.3.3]{CA4}. However, recently, it was shown that the free relativized cylindric set algebras are not atomic, but still they contain some atoms \cite{myphd}. Investigating the non-atomicity of free algebras in algebraic logic is an ongoing research project punctuated by many deep results and challenges. See section~\ref{newsection} for more details about the current status of this project.

In this paper, we consider diagonal-free versions of relativized cylindric set algebras $\ds_{\alpha}$. We prove that $\a{Fr}_{\beta}\ds_{\alpha}$ is atomless whenever $\alpha\geq 2$ and $\beta\geq 1$. Considering this in line with the results in \cite{myphd} gives us some information about the differences between guarded logics with identity and guarded logics without identity as a privileged logical symbol. The methods we use here are similar to the ones in \cite{myphd}, but applied in new directions. As a strength sign of our methods, we collect other important results too, e.g.,, the decidability of the equational theory of $\ds_{\alpha}$, which is proved by these methods.

Diagonal-free relativized cylindric set algebras correspond to first order logic without equality on general assignment models. We will discuss this correspondence and the applications of our results in section~\ref{newsection}. The interest for the study of languages without equality has its origin in the works of W. Blok and D. Pigozzi, \cite{pigozzi1}, \cite{pigozzi2} and \cite{pigozzi3}. Several developments and interesting results in this direction have been made, c.f., e.g.,, \cite{elg1}, \cite{dell1}, \cite{dell2}, \cite{dell3} and \cite{elg2} 
\section{Preliminaries and main results}
Recall the basic concepts of universal algebra from the literature, see, e.g.,, \cite{universal}. Let $\mathsf{K}$ be any class of algebras of the same similarity type, then $\mathbf{I}\mathsf{K}$, $\mathbf{S}\mathsf{K}$,  $\mathbf{P}\mathsf{K}$ and $\mathbf{H}\mathsf{K}$ are the classes that consist of the isomorphic copies, subalgebras, (isomorphic copies of) direct products and homomorphic images, respectively, of the members of $\mathsf{K}$. Let $X$ be any set, then $\a{Fr}_X\mathsf{K}$ is the free algebra of the class $\mathsf{K}$ generated by the free variables in $X$. Throughout this paper, we fix an ordinal $\alpha\geq 2$.

We start with the following basic notions. For every $i\in \alpha$ and every two sequences $f,g$ of length $\alpha$, we write $f\equiv_ig$ if and only if $g= f^u_i$ for some $u$, where $f^u_i$ is the sequence which is like $f$ except that it's value at $i$ equals $u$. Let $V$ be an arbitrary set of sequences of length $\alpha$. For each $i\in \alpha$, let $C_i^{[V]}$ be the mapping from $\mathcal{P}(V)$ into $\mathcal{P}(V)$ defined as follows: for any $X\subseteq V$, $$C_i^{[V]}X=\{f\in V:(\exists g\in X)f\equiv_ig\}.$$ This is called the $V$-cylindrification in the direction $i$. When no confusion is likely, we merely omit the superscript $[V]$ from the above defined object.

\begin{defn}The class of all \textbf{representable relativized diagonal free algebras of dimension $\alpha$}, denoted by $\ds_{\alpha}$, is defined to be the class that consists of all isomorphic copies of the subalgebras of the (full) algebras of the form, $$\a{P}(V)\myeq\langle \mathcal{P}(V), \cap,\cup,\setminus,\emptyset, V, C_i^{[V]}\rangle_{i\in \alpha},$$
where $V$ is a non-empty set of sequences of length $\alpha$ and $\mathcal{P}(V)$ is the family of all subsets of $V$. In other words, $\ds_\alpha=\mathbf{I}\mathbf{S}\{\a{P}(V): \emptyset\not=V\subseteq{^{\alpha}U}\text{ for some set }U\}$. For every $\a{A}\subseteq\a{P}(V)$, the set $V$ is called the unit of $\mathfrak{A}$, while the smallest set $U$ that satisfies $V\subseteq {^{\alpha}U}$ is called the base of $\mathfrak{A}$.
\end{defn}

\begin{prop}\label{variety}
The class $\ds_{\alpha}$ is a variety.
\end{prop}
\begin{proof}[Sketch of the proof]
We need to show that $\ds_{\alpha}$ is closed under $\mathbf{S}$, $\mathbf{P}$ and $\mathbf{H}$. By definition, it is clear that $\ds_{\alpha}$ is closed under forming subalgebras $\mathbf{S}$.

{\bf{$\ds_{\alpha}$ is closed under $\mathbf{P}$:}} 
Let $\a{A}$ and $\a{B}$ be two algebras in $\ds_{\alpha}$. We show that their direct product $\a{A}\times\a{B}\in\ds_{\alpha}$. The same method can be applied to show that the direct product of any set of algebras in $\ds_{\alpha}$ is an element of $\ds_{\alpha}$. By definition, there are two non-empty sets (of sequences of length $\alpha$) $V_1$ and $V_2$ such that $\a{A}$ and $\a{B}$ are isomorphic to subalgebras of the full algebras $\a{P}(V_1)$ and $\a{P}(V_2)$, respectively. Let $U_1$ and $U_2$ be the bases of $\a{P}(V_1)$ and $\a{P}(V_2)$, respectively. We may assume that $U_1\cap U_2=\emptyset$. Now, we need to show that $\a{P}(V_1)\times\a{P}(V_2)\cong\a{P}(V_1\cup V_2)$.

Define the map $\psi:\mathcal{P}(V_1)\times\mathcal{P}(V_2)\rightarrow\mathcal{P}(V_1\cup V_2)$ as follows. For each $X\subseteq V_1$ and each $Y\subseteq V_2$, let $\psi(X,Y)=X\cup Y$. It is not hard to see that $\psi$ is a homomorphism because $U_1\cap U_2=\emptyset$. 
It remains to prove that $\psi$ is an injection. It is enough to show that the kernel of $\psi$ is $\{\emptyset\}$, which is clear by the definition of $\psi$. Therefore, $\a{A}\times\a{B}$ is isomorphic to a subalgebra of $\a{P}(V_1\cup V_2)$.

{\bf{$\ds_{\alpha}$ is closed under $\mathbf{H}$:}} By the first homomorphism theorem, we know that every homomorphic image of an algebra is isomorphic to a quotient of this algebra. Thus, it is enough to prove that every quotient algebra of a member of $\ds_{\alpha}$ is a member of $\ds_{\alpha}$. Suppose that $\a{A}$ is a subalgebra of $\a{P}(V)$ for some non-empty set $V$ (of sequences of length $\alpha$). Suppose that $\Theta$ is a congruence relation on $\a{A}$. For every $X\in A$, let $[X]=\{Y\in A:(X,Y)\in\Theta\}$ and let $$\bigcup[X]=\{y\in V:y\in Y\text{ for some }Y\in[X]\}.$$ Let $V'=V\setminus (\bigcup[\emptyset])$. We prove that $\a{A}/\Theta$ is embeddable into the full algebra $\a{P}(V')$. Define $\psi:A/\Theta\rightarrow\mathcal{P}(V')$ as follows. For each $X\in A$, let $\psi([X])=(\bigcup[X])\cap V'$. The fact that $\Theta$ is a congruence on $\a{A}$ implies that $\psi$ is an injective homomorphism.
%
Therefore, $\a{A}/\Theta\in\ds_{\alpha}$.
\end{proof}

We assume familiarity with the basic notions of the theory of cylindric algebras, e.g., atoms, zero-dimensional elements, etc. The definitions of such notions can be found in \cite{hmt1} and/or \cite{hmt2}. We shall mention that several general theorems from literature can be applied to obtain results concerning $\ds_{\alpha}$. For example, Theorem~\ref{main1} below (at least for the case when $\alpha$ is finite) may follow as a consequence of \cite[Theorem 9.4]{monk}, \cite[Theorem 4.2]{nem86} and \cite[Theorem 5.3.5]{marx}. 

In the present paper, we give direct proofs of these facts (for finite and infinite $\alpha$'s). Our technique also leads to some new important results, see Theorem~\ref{main2} and Theorem~\ref{main3}.
\begin{mthm}\label{main1}
The variety $\ds_{\alpha}$ enjoys each of the following:
\begin{enumerate}[(1)]
\item Finite schema axiomatizability (finite axiomatizability if $\alpha$ is finite).
\item Finite base property, i.e.,, generated by its algebras whose base is finite.
\item Decidable equational theory.
\item Generated by its locally finite dimensional algebras, for the case when $\alpha$ is infinite.
\item Super amalgamation property.
\end{enumerate}
\end{mthm}

\begin{mthm}\label{main2}
Let $X$ be any set, we have the following:
\begin{itemize}
\item If $X=\emptyset$ then $\f{X}{\alpha}$ is a two-element algebra, hence it is atomic.
\item If $X\not=\emptyset$ then the free algebra $\f{X}{\alpha}$ is atomless.
\end{itemize}
\end{mthm}
\begin{mthm}\label{main3}
Let $X$ be any set. The only zero-dimensional elements in the free algebra $\f{X}{\alpha}$ are the zero and the unit.
\end{mthm}

To prove the above theorems, we use the normal forms defined in \cite{mysh}, these are generalizations of the normal forms introduced by J. Hintikka \cite{hintikka}. We show that, for each satisfiable normal form in $\ds_{\alpha}$, there is an algebra in $\ds_{\alpha}$, whose base is finite, that witnesses the satisfiability of this form. 
\section{Axioms and normal forms}
Here, we give an equational characterization for the class $\ds_{\alpha}$. We get this characterization by deleting the cylindrifiers-commutativity axiom from the axioms defining the diagonal free cylindric algebras defined by A. Tarski \cite[Definition 1.1.2]{hmt1}.
\begin{defn}Let $\d_{\alpha}$ be the class of all \textbf{relativized diagonal free algebras of dimension $\alpha$}, i.e., the class consists of all algebras $\a{A}=\langle A, \cdot, +,-,0,1,c_i\rangle_{i\in \alpha}$, that satisfy the following equations for every $i\in \alpha$.
\begin{enumerate}[(Ax 0)]
\item[(Ax 0)] The set of equations characterizing Boolean algebras for $\cdot,+,-,0,1$.
\item[(Ax 1)] The set of equations defining $c_i$ as an additive, closure and complemented operator:
\begin{enumerate}
\item[(Ax 1a)] $c_{i}0=0$.
\item[(Ax 1b)] $x+c_{i}x=c_{i}x$.
\item[(Ax 1c)] $c_{i}(x\cdot c_iy)=c_{i}x\cdot c_{i}y$.
\end{enumerate}
\end{enumerate}
\end{defn}
Note that $\ds_{\alpha}\subseteq\d_{\alpha}$. It is easy to check that each $\a{A}\in\ds_{\alpha}$ satisfies the above axioms. Later, we will prove that the above is actually a characterization of the class $\ds_{\alpha}$, i.e., $\ds_{\alpha}=\d_{\alpha}$. Let $X$ be any set of variables, then $T_{\alpha}(X)$ is defined to be the set of all terms in the signature of $\d_{\alpha}$ that are built up from variables in $X$.

Now, we define normal forms in the signature of $\d_{\alpha}$. Then, we will show that each term in this signature can be rewritten equivalently as a Boolean joint of these normal forms. Let $\prod$ and $\sum$ be the grouped versions of $\cdot$ and $+$ respectively. Empty product and empty sum are defined to be $1$ and $0$ respectively. Let $X$ be a set of variables and let $T\subseteq T_{\alpha}(X)$ be a finite set of terms. 
Let $n\in\alpha+1$ be a finite ordinal and let $\beta\in{^T\{-1,1\}}$. Define 
$$C_n(T)\myeq\{c_i\tau:\tau\in T,i\in n\} \ \  \text{ and } \ \ T^{\beta}\myeq\prod\{\tau^{\beta}:\tau\in T\},$$
where, for every $\tau\in T$, $\tau^{\beta}=\tau$ if $\beta(\tau)=1$ and $\tau^{\beta}=-\tau$ otherwise.
%
\begin{defn}\label{forms}
Let $X$ be a finite set and let $n\in\alpha+1$ be finite ordinal such that $n\geq 2$. Let $k\in\omega$, we define the following inductively.
\begin{enumerate}
\item[-] Normal forms of degree $0$: $F_{0}(X;n)\myeq\{X^{\beta}:\beta\in{^{X}\{-1,1\}}\}$.
\item[-] Normal forms of degree $k+1$: \begin{center}$F_{k+1}(X;n)\myeq\{X^{\beta}\cdot (C_n(F_k(X;n)))^{\alpha}:\beta\in{^{X}\{-1,1\}}\text{ and }\alpha\in{^{C_n(F_k(X;n))}\{-1,1\}}\}$.\end{center}
\item[-] All normal forms, $F(X;n)\myeq\bigcup_{k\in\omega} F_k(X;n)$.
\end{enumerate}
\end{defn}
The prove of the following theorem can be found in \cite[Lemma 4.9 and Theorem 4.10]{mysh}. 
\begin{thm}\label{andreka}Let $k\in\omega$ and $n\in\alpha+1$ be finite ordinals such that $n\geq 2$. Let $X$ be a finite set of variables. Then the following are true:
\begin{enumerate}[(i)]
\item\label{andc1} $\d_{\alpha}\models\sum F_k(X;n)=1$.
\item\label{andc2} For every $\tau,\sigma\in F_k(X;n)$, if $\tau\not=\sigma$ then $\d_{\alpha}\models\tau\cdot \sigma=0$. 
\item\label{andc4} Let $\tau\in T_n(X)$ \footnote{Again, $T_n(X)$ is the set of all terms in the signature of $\d_n$ that are built up from variables in $X$.} be such that $\d_{\alpha}\not\models\tau=0$. Then there is a finite ordinal $q\in\omega$ and a non-empty finite set $S\subseteq F_q(X;n)$ of normal forms of degree $q$ such that $\d_\alpha\models\tau=\sum S$.
\end{enumerate}
\end{thm}
Note that (i) and (ii) of the above theorem state that $F_k(X;n)$ forms a partition of the unit. The following definition introduces some notations that will be used in the proceeding sections.
\begin{defn}
Let $k\in\omega$ and $n\in\alpha+1$ be finite ordinals such that $n\geq 2$. Let $X$ be a finite set of variables. Let $\beta\in{^{X}\{-1,1\}}$ and $\alpha\in{^{C_n(F_k(X;n))}\{-1,1\}}$. For each $i\in n$, define
\begin{center}
$\sub_i(X^{\beta}\cdot(C_n(F_k(X;n)))^{\alpha})\myeq\{\sigma\in F_k(X;n):\alpha(c_i\sigma)=1\}$, and
\end{center}
\begin{center}$\color(X^{\beta})\myeq \color(X^{\beta}\cdot(C_n(F_k(X;n)))^{\alpha})\myeq\{\sigma\in X:\beta(\sigma)=1\}$.\end{center}
\end{defn}
\section{Finite schema axiomatizability, decidability, etc}
In this section, we fix finite ordinals $k\in\omega$ and $n\in\alpha+1$ such that $n\geq 2$, and we fix finite set $X$. Consider any normal form $\tau\in F_k(X;n)$. We will construct a unit $V$, on a finite base, such that $$\d_{\alpha}\not\models\tau=0\iff\a{P}(V)\not\models\tau=0.$$ 
We do this inductively by constructing a finite sequence $V_0\subseteq V_1\subseteq \cdots \subseteq V_k$ (of length $k+1$), and then we let $V\myeq V_k$ be the desired unit. Throughout the construction, whenever we add an element $e\in V$, we label it by some normal form $\tg(e)\in F_0(X;n)\cup\cdots\cup F_k(X;n)$. While constructing $V$, our target is to guarantee that each element satisfies its label. 

To start, let $U'$ be an infinite set and let $t\not\in U'$ be any entity. For any elements $u_0,\ldots,u_{n-1}\in U'$, by writing $\bar{u}=(u_0,\ldots,u_{n-1},\bar{t})$ we mean the sequence, of length $\alpha$, defined as follows: For each $i\in n$, $\bar{u}(i)=u_i$. For each $i\in\alpha\setminus n$, $\bar{u}(i)=t$. The sequence $\bar{t}$ is called the tail of the desired unit $V$.

{\bf{\underline{Constructing $V_0$}:}}

Let $U_0\myeq\{v_0,\ldots,v_{n-1}\}\subseteq U'$ be such that $\lvert U_0\rvert=n$. Let $V_0=V_0^0=\cdots=V_0^{n-1}\myeq\{(v_0,\ldots,v_{n-1},\bar{t})\}$. Define the label of the unique element in $V_0$ as follows: $\tg(v_0,\ldots,v_{n-1},\bar{t})=\tau$. 

Suppose that, for some $l\in k$, we are given the finite sets $U_l\subseteq U'$, $V_l,V_l^0,\cdots,V_l^{n-1}\subseteq {^n{U_l}}\times\{\bar{t}\}$. Also, assume that we are given the labels of the elements in $V_l$.

{\bf{\underline{Constructing $V_{l+1}$}:}}

For every $i\in n$, every $j\in n\setminus\{i\}$ and every $\bar{v}\in V_l^{j}$, create an injective function $$\psi_{\bar{v}}^i:\sub_i(\tg(\bar{v}))\rightarrow U'\setminus U_l,$$ such that the ranges $(\psi^i_{\bar{v}})^*$ of all of those functions are pairwise disjoint and $U'\setminus U_{l+1}$ is still infinite, where $U_{l+1}\myeq\bigcup\{(\psi^i_{\bar{v}})^*:i\in n, \bar{v}\in V_l^j\text{ for some }j\in n,j\not=i\}\cup U_l$. Now, for every $i\in n$, let
$$V_{l+1}^i\myeq\{\bar{v}_{i}^{u}:\bar{v}\in V_l^j\text{ for some }j\in n, j\not=i, \text{ and } u\in (\psi^i_{\bar{v}})^*\}.$$
Let $V_{l+1}\myeq V_l\cup V_{l+1}^0\cup\cdots\cup V_{l+1}^{n-1}$. We extend the labels as follows: Let $i\in n$, $j\in n\setminus\{i\}$, $\bar{v}\in V^j_l$ and $\sigma\in \sub_i(\tg(\bar{v}))$. Suppose that $u=\psi^i_{\bar{v}}(\sigma)$, define $\tg(\bar{v}_i^u)\myeq \sigma$.

\begin{figure}[!ht]
\centering
\begin{minipage}{\textwidth}
\centering
\begin{tikzpicture}
\draw (0,3.4) node {};
\draw (0,-1.4) node {};
\draw (0, 0) ellipse (0.4 and 0.2);
\draw (0,-0.4) node {$V_0$};
\draw (-2.5, 1) ellipse (0.8 and 0.3);
\draw (-3.6,1) node {$V_1^0$};
\draw (2.5, 1) ellipse (0.8 and 0.3);
\draw (3.6,1) node {$V_1^2$};
\draw (0,0) node {$\bullet$};
\draw (0.5,1) node {$\bullet$};
\draw (-0.5,1) node {$\bullet$};
\draw (0,0) -- (-0.5,1);
\draw (0,0) -- (0.5,1);
\draw (-0.5,1) -- (0.5,1);
\draw (0, 1) ellipse (0.8 and 0.3);
\draw (0,1.5) node {$V_1^1$};
\draw (-2,1) node {$\bullet$};
\draw (2,1) node {$\bullet$};
\draw (0,0) -- (-2,1);
\draw (0,0) -- (2,1);
\draw (-3,1) node {$\bullet$};
\draw (3,1) node {$\bullet$};
\draw (0,0) -- (-3,1);
\draw (0,0) -- (3,1);
\draw (-3,1) -- (-2,1);
\draw (3,1) -- (2,1);
\draw (4,2.5) node {$\bullet$};
\draw (5,2.5) node {$\bullet$};
\draw (2,2.5) node {$\bullet$};
\draw (1,2.5) node {$\bullet$};
\draw (3,1) -- (4,2.5);
\draw (3,1) -- (5,2.5);
\draw (3,1) -- (2,2.5);
\draw (3,1) -- (1,2.5);
\draw (5,2.5) -- (4,2.5);
\draw (2,2.5) -- (1,2.5);
\draw (4.5, 2.5) ellipse (0.8 and 0.3);
\draw (5.8,2.5) node {$\subseteq V_2^1$};
\draw (0.2,2.5) node {$V_2^0\supseteq$};
\draw (1.5, 2.5) ellipse (0.8 and 0.3);
\end{tikzpicture}
\end{minipage}
\end{figure}

{\bf{\underline{The desired algebra:}}}

Finally, let $V\myeq V_k$ and $U\myeq U_k$. Remember $V\subseteq {^n U}\times\{\bar{t}\}$. We call $U$ the actual base of $V$.

Note that, for each $\bar{v}\in V$ and each $1\leq l\leq k$, we have 
\begin{equation}\label{deg}
\bar{v}\in V_0\iff \tg(\bar{v})=\tau\in F_k(X;n) \ \ \text{ and } \ \ 
\bar{v}\in V_l\setminus V_{l-1}\iff \tg(\bar{v})\in F_{k-l}(X;n).
\end{equation}
Define the evaluation, $\ev:X\rightarrow \mathcal{P}(V)$, of free variables into $\a{P}(V)$ as follows: For each $x\in X$, let $\ev(x)\myeq\{\bar{v}\in V:x\in \color(\tg(\bar{v}))\}$. For every $\bar{v}\in V$ and every term $\sigma\in T_{\alpha}(X)$, we write $(V,\ev,\bar{v})\models\sigma$ if and only if $\bar{v}$ is in the interpretation of the term $\sigma$ in the full algebra $\a{P}(V)$, under the evaluation $\ev$. Now, we prove that $\a{P}(V)$ is as desired.
\begin{lem}\label{7abibi}
$\d_{\alpha}\not\models\tau=0$ if and only if $\a{P}(V)\not\models\tau=0$.
\end{lem} 
\begin{proof}
Note that $\ds_{\alpha}\subseteq \d_{\alpha}$, so the direction ($\Leftarrow$) is trivial. Now, we prove the non-trivial direction ($\Rightarrow$). Suppose that $\d_{\alpha}\not\models\tau=0$. Then by (Ax 1a), it follows that $\d_{\alpha}\not\models \tg(\bar{v})=0$ for each element  $\bar{v}\in V$. For each $\bar{v}\in V$ and each $h\leq k$, we define $\tg_h(\bar{v})$ as follows:

Suppose that $l\leq k$ is the smallest number for which $\bar{v}\in V_l$. Remember the facts $\d_{\alpha}\not\models \tg(\bar{v})=0$ and $\tg(\bar{v})\in F_{k-l}(X;n)$ (by \eqref{deg}). If $h\geq k-l$ then we define $\tg_h(\bar{v})\myeq \tg(\bar{v})$. Suppose $h< k-l$, then by Theorem~\ref{andreka} there is a unique normal form $\sigma\in F_h(X;n)$ such that $\d_{\alpha}\models \tg(\bar{v})\leq \sigma$, so in this case we define $\tg_h(\bar{v})\myeq\sigma$.

To finish, it is enough to prove the following. For each $\bar{v}\in V$ and each $h\leq k$,
\begin{equation}\label{eq}
(V,\ev,\bar{v})\models \tg_h(\bar{v}).
\end{equation}
We use induction on $h$. The choice of the evaluation guarantees that \eqref{eq} is true for every $\bar{v}\in V$ when $h=0$. Suppose that \eqref{eq} holds for every $\bar{v}\in V$, for some $h\in k$. Let $\bar{v}\in V$, we need to show that $(V,\ev,\bar{v})\models \tg_{h+1}(\bar{v})$. Let $l\leq k$ be the smallest number for which $\bar{v}\in V_l$. Thus, by \eqref{deg}, we have $\tg(\bar{v})\in F_{k-l}(X;n)$. If $h\geq k-l$, then $\tg_{h+1}(\bar{v})=\tg_h(\bar{v})$ and, by the induction hypothesis, we are done. So, suppose that $h<k-l$. We consider two cases.
\begin{enumerate}[(I)]
\item The first case is when $l=0$. In this case, $\tg(\bar{v})=\tau\in F_k(X;n)$. By the choice of the evaluation $\ev$, it is easy to check that
\begin{equation}\label{c1x}
(\forall x\in X) \ \ \ [(V,\ev,\bar{v})\models x \iff x\in \color(\tg_{h+1}(\bar{v}))].
\end{equation}
Let $i\in n$ and let $\sigma\in F_h(X;n)$. We need to prove the following.
\begin{equation}\label{c1s}
\sigma\in \sub_i(\tg_{h+1}(\bar{v}))\iff (V,\ev,\bar{v})\models c_i\sigma.
\end{equation}
Suppose that $\sigma\in \sub_i(\tg_{h+1}(\bar{v}))$. By Theorem~\ref{andreka} (iii), there is a finite set $S\subseteq F_{k-1}(X;n)$ such that $\d_{\alpha}\models\sum S=\sigma$. We claim that there is $\sigma'\in S$ such that $\sigma'\in \sub_i(\tg(\bar{v}))$. Suppose towards a contradiction that $\d_{\alpha}\models \tg(\bar{v})\leq -c_i\sigma'$ for every $\sigma'\in S$. Then, by the additivity of the operator $c_i$, it follows that $\tg(\bar{v})\leq -c_i \sum S=-c_i\sigma$. This contradicts the facts that $\sigma\in \sub_i(\tg_{h+1}(\bar{v}))$ and $\d_{\alpha}\models\tau\not=0$. Thus, there exists $\sigma'\in \sub_i(\tg(\bar{v}))$ such that $\d_{\alpha}\models\sigma'\leq \sigma$. By the construction of $V$, there exists $\bar{u}\in V_1^i$ such that $\tg(\bar{u})=\sigma'$ and $\bar{v}\equiv_i\bar{u}$. Hence, by induction hypothesis, $(V,\ev,\bar{u})\models \tg_h(\bar{u})=\sigma$. Therefore, $(V,\ev,\bar{v})\models c_i\sigma$. 

Conversely, let $\bar{u}\in V$ be such that $\bar{v}\equiv_i\bar{u}$ and $(V,\ev,\bar{u})\models\sigma$. Suppose that $\bar{u}=\bar{v}$, then by induction hypothesis and Theorem~\ref{andreka} (ii) we have $\tg_h(\bar{v})=\sigma$. Again, by Theorem~\ref{andreka}, $\d_{\alpha}\models\tau\leq\sigma$. Thus, by axiom (Ax 1b), $\sigma\in \sub_j(\tg_{h+1}(\bar{v}))$ as desired. Suppose that $\bar{u}$ and $\bar{v}$ are different. By the construction, there exists $\sigma'\in \sub_i(\tg(\bar{v}))$ such that $\tg(\bar{u})=\sigma'$. Then, by the induction hypothesis, we must have $\sigma=\tg_h(\bar{u})$, i.e., $\d_{\alpha}\models \sigma'\leq\sigma$. Thus, $\d_{\alpha}\not\models \tg(\bar{v})\cdot c_i\sigma=0$. But $\d_{\alpha}\models \tg(\bar{v})\leq \tg_{h+1}(\bar{v})$, hence $\d_{\alpha}\not\models \tg_{h+1}(\bar{v})\cdot c_i\sigma=0$. This can happen only if $\sigma\in \sub_i(\tg_{h+1}(\bar{v}))$.

Thus, we have shown that \eqref{c1s} is true for every $i\in n$ and every $\sigma\in F_h(X;n)$. Therefore, by \eqref{c1x} and \eqref{c1s}, we have $(V,\ev,\bar{v})\models \tg_{h+1}(\bar{v})$, as desired.
\item Now, suppose that $l\not=0$. Again the choice of the evaluation guarantees the following. 
\begin{equation}\label{c2x}
(\forall x\in X) \ \ \ [(V,\ev,\bar{v})\models x \iff x\in \color(\tg_{h+1}(\bar{v}))].
\end{equation}
Suppose that $\bar{v}\in V_l^j$ for some $j\in n$. Let $i\in n$ be such that $i\not=j$. Then, by a similar argument to the one used in the above item, one can see that 
\begin{equation}\label{c2si}
(\forall \sigma\in F_h(X;n)) \ \ \ [\sigma\in \sub_i(\tg_{h+1}(\bar{v}))\iff (V,\ev,\bar{v})\models c_i\sigma].
\end{equation}
By the construction of $V$ and since $l\not=0$, there exists an element $\bar{w}\in V_{l-1}$ such that $(l-1)$ is the smallest number for which $\bar{w}\in V_{l-1}$, $\bar{w}\equiv_j\bar{v}$ and $\tg(\bar{v})\in \sub_j(\tg(\bar{w}))$. Thus, by axioms (Ax 1a) and (Ax 1c), and Theorem~\ref{andreka}, it follows that $\sub_j(\tg_{h+1}(\bar{v}))=\sub_j(\tg_{h+1}(\bar{w}))$. Since $\bar{v}\in V_l^j$, then $\bar{w}\not\in V^j_{l-1}$. Now, by \eqref{c2si} it follows that
\begin{equation}\label{c2spj}
(\forall \sigma\in F_h(X;n)) \ \ \ [(V,\ev,\bar{w}) \models c_j\sigma \iff \sigma\in \sub_j(\tg_{h+1}(\bar{w}))]
\end{equation}
Hence, for every $\sigma\in F_h(X;n)$, we have
\begin{equation}\label{c2sj}
 \sigma\in \sub_j(\tg_{h+1}(\bar{v}))=\sub_j(\tg_{h+1}(\bar{w}))\iff (V,\ev,\bar{w}) \models c_j\sigma\iff (V,\ev,\bar{v}) \models c_j\sigma.
\end{equation}
Finally, \eqref{c2x}, \eqref{c2si} and \eqref{c2sj} imply that $(V,\ev,\bar{v})\models \tg_{h+1}(\bar{v})$, as desired.
\end{enumerate}
Thus, by the principle of mathematical induction, we have shown that \eqref{eq} holds for each $h\leq k$ and each $\bar{v}\in V$. Now, let $\bar{v}$ be the unique node in $V_0$. Note that $\tg_{k}(\bar{v})=\tau$. Hence, $(V,\ev,\bar{v})\models\tau$. Therefore, $\a{P}(V)\not\models\tau=0$ and we are done.
\end{proof}
Now, the first three items of Theorem~\ref{main1} are direct consequences of Lemma~\ref{7abibi}. The super amalgamation property follows from Lemma~\ref{7abibi} together with \cite[Theorem 5.3.5]{marx}.
\begin{proof}[Proof of Theorem~\ref{main1}]\
\begin{enumerate}
\item[(1)] To prove the finite schema axiomatizability, it is enough to prove that $\ds_{\alpha}=\d_{\alpha}$. Let $Y$ be any set of variables (finite or infinite) and let $\tau\in T_{\alpha}(Y)$. Then, 
$$\d_{\alpha}\models\tau=0\iff\ds_{\alpha}\models\tau=0.$$ 
The implication $\Longrightarrow$ follows from the fact that $\ds_{\alpha}\subseteq\d_{\alpha}$. For the other direction, suppose that $\d_{\alpha}\not\models\tau=0$. Let $m\in\omega$ and let $Z\subseteq Y$ be a finite set such that $\tau\in T_m(Z)$. We can assume that $m\geq 2$. By Theorem~\ref{andreka} \eqref{andc4}, there are $k\in\omega$ and $\sigma\in F_k(Z;m)$ such that $\d_{\alpha}\models \sigma\leq\tau$ and $\d_{\alpha}\not\models\sigma=0$. The fact that $\ds_{\alpha}\subseteq \d_{\alpha}$ implies that $\ds_{\alpha}\models\sigma\leq\tau$. Moreover, by Lemma~\ref{7abibi} we can say that $\ds_{\alpha}\not\models\sigma=0$. Thus, we have $\ds_{\alpha}\not\models\tau=0$ and the implication $\Longleftarrow$ is established. Therefore, we can also deduce that $\a{Fr}_Y\d_{\alpha}\cong \f{Y}{\alpha}$.

Now, we need to show that $\d_{\alpha}\subseteq\ds_{\alpha}$. Let $\a{A}\in\d_{\alpha}$. By the universal mapping property, there is an onto homomorphism $f:\a{Fr}_A\d_{\alpha}\rightarrow\a{A}$. In other words, $\a{A}$ is a homomorphic image of $\a{Fr}_A\d_{\alpha}$. Thus, $\a{A}$ is a homomorphic image of $\mathfrak{Fr}_A\ds_{\alpha}$ too. But $\ds_{\alpha}$ is a variety, so it contains all its free algebras and it is closed under $\mathbf{H}$. Therefore, $\a{A}\in\ds_{\alpha}$ as desired. 
\item[(2)] The finite base property follows immediately from Lemma~\ref{7abibi}, Theorem\ref{andreka} and (1).
\item[(3)] It is known that the finite schema axiomatizability and the finite base property imply the decidability of the equational theory, see, e.g., \cite{fmpd}.
\item[(4)] All the algebras constructed in this section are locally finite dimensional. Thus $\ds_{\alpha}$ is generated by its locally finite dimensional algebras.
\item[(5)] We showed that $\ds_{\alpha}$ is characterized by positive equations, so it is canonical variety. Thus, super amalgamation property follows from \cite[Theorem 5.3.5]{marx}.
\end{enumerate}
\end{proof}
\section{Free algebras: atoms and zero-dimensional elements}
In this section we give the proof of Theorem~\ref{main2}. We start with the following.
\begin{thm}
The free algebra $\f{\emptyset}{\alpha}$ is a two elements algebra, hence it is atomic.
\end{thm}
\begin{proof}
Straightforward since any finite Boolean algebra is atomic.
\end{proof}
\begin{thm}
Let $X$ be an infinite set, the free algebra $\f{X}{\alpha}$ is atomless.
\end{thm}
\begin{proof}(Essentially due to D. Pigozzi \cite[2.5.13]{hmt1})
Let $\tau\in T_{\alpha}(X)$ be such that $\ds_{\alpha}\not\models\tau=0$. We show that $\tau$ is not an atom in $\f{X}{\alpha}$. Note that there is a finite $Y\subseteq X$ such that $\tau\in T_{\alpha}(Y)$. Let $y\in X\setminus Y$ and let $\a{B}\in\ds_{\alpha}$ be such that $\a{B}\not\models\tau=0$. By the universal mapping property, there are homomorphisms $f:\f{X}{\alpha}\rightarrow\a{B}$ and $g:\f{X}{\alpha}\rightarrow\a{B}$ such that $f(z)=g(z)$, for all $z\in Y$, while $f(y)=1$ and $g(y)=0$. Then, $f(\tau)=g(\tau)$. Hence, $f(\tau\cdot y)=g(\tau\cdot -y)=\tau^{\a{B}}\not=0$. So, $\ds_{\alpha}\not\models\tau\cdot y=0$ and $\ds_{\alpha}\not\models \tau\cdot-y=0$. Thus, $\tau$ can not be an atom in $\f{X}{\alpha}$. 
\end{proof}
To prove the remaining part of Theorem~\ref{main2}, we need to prove the following lemma. Let $k\in\omega$ and $n\in\alpha+1$ be finite ordinals such that $n\geq 2$, and let $X$ be a non-empty finite set. Let $\tau\in F_k(X;n)$. Recall the unit $V$ constructed in the previous section that witnesses the satisfiability of $\tau$.

\begin{figure}[!ht]
\centering
\begin{minipage}{\textwidth}
\centering
\begin{tikzpicture}
\draw (0,3.5) node {$V$};
\draw (0,0) ellipse (2 and 3);
\draw (0,-2.8) node {$\bullet$};
\draw (-0.3,-2.7) node {$\bar{v}_0$};
\draw (1,-2) node {$\bullet$};
\draw (1.3,-1.9) node {$\bar{v}_1$};
\draw (-1,-1) node {$\bullet$};
\draw (-1.3,-0.9) node {$\bar{v}_2$};
\draw (1,0) node {$\bullet$};
\draw (-1,1) node {$\bullet$};
\draw (1,2) node {$\bullet$};
\draw (0,2.8) node {$\bullet$};
\draw (-0.3,2.7) node {$\bar{v}_k$};
\draw (0,-2.8) -- (1,-2);
\draw (1,-2) -- (-1,-1);
\draw (-1,-1) -- (1,0);
\draw[dashed] (1,0) -- (-1,1);
\draw[dashed] (-1,1) -- (1,2);
\draw (1,2) -- (0,2.8);
\end{tikzpicture}
\end{minipage}
\end{figure}

\begin{lem}\label{seq}Suppose that $k$ is even. There is a sequence $\bar{v}_0,\ldots,\bar{v}_{k}\in V$ such that:
\begin{enumerate}[(1)]
\item For every $h\in k+1$, $\tg(\bar{v}_h)\in F_{h}(X;n)$. In particular, $\tg(\bar{v}_k)=\tau$.
\item For every $h\in k$: If $h$ is odd then $\bar{v}_h\equiv_0\bar{v}_{h+1}$. If $h$ is even then $\bar{v}_h\equiv_1\bar{v}_{h+1}$.
\end{enumerate}
\end{lem}
\begin{proof}
Let $\bar{v}_k$ be the only element in $V_0$. If $k=0$, then we are done. So let us suppose $k\not=0$. Now since $\d_{\alpha}\models\tau\not=0$ then there exists a normal form $\tau_1\in F_{k-1}(X;n)$ such that $\d_{\alpha}\models\tau\leq \tau_1$. 
%
By axiom (Ax 1b), we have $\d_{\alpha}\models\tau\leq c_0\tau_1$. Hence, $\tau_1\in \sub_0(\tau)$. Let $\bar{v}_{k-1}$ be the unique element in $V_1^0$ with $\bar{v}_k\equiv_0\bar{v}_{k-1}$ and $\tg(\bar{v}_{k-1})=\tau_1$. If $k=1$, then we are done. Suppose that $k>1$, since $\d_{\alpha}\models\tau_1\not=0$ then there exists unique $\tau_2\in F_{k-2}(X;n)$ such that $\d_{\alpha}\models\tau_1\leq\tau_2$. Again, by axiom (Ax 1b), we have $\tau_2\in \sub_1(\tau_1)$. Let $\bar{v}_{k-2}$ be the unique element in $V_2^1$ with $\bar{v}_{k-1}\equiv_1\bar{v}_{k-2}$ and $\tg(\bar{v}_{k-2})=\tau_2$. Continue in this manner, we get the desired sequence.
\end{proof}
\begin{thm}
Let $X$ be a non-empty finite set, the free algebra $\f{X}{\alpha}$ is atomless. 
\end{thm}
\begin{proof}
Let $\sigma\in T_{\alpha}(X)$ be such that $\ds_{\alpha}\not\models\sigma=0$. We need to show that $\sigma$ is not an atom in $\f{X}{\alpha}$. By Theorem~\ref{andreka} and Theorem~\ref{main1} (1), there are finite $n\in\alpha+1$ ($n\geq 2$), finite $k\in\omega$ and non-empty finite set $S\subseteq F_k(X;n)$ such that $\ds_{\alpha}\models\sigma=\sum S$.
Thus, there is $\tau\in F_k(X;n)$ such that $\ds_{\alpha}\models0\not=\tau\leq\sigma$. 
%
%
So, to prove that $\sigma$ is not an atom in the free algebra $\f{X}{\alpha}$, it is enough to prove that $\tau$ is not an atom in $\f{X}{\alpha}$. We do this through the following steps. Without loss  of generality we can assume that $k$ is an even number. 

{\bf{Step 1:}}
Given the normal form $\tau$, construct the unit $V$, the actual base $U$, the tail $\bar{t}$ and the evaluation $\ev$ as constructed in the previous section. Recall that we have
\begin{equation}\label{original}
(\forall\bar{v}\in V) \ \ \ (V,\ev,\bar{v})\models \tg(\bar{v}).
\end{equation}

{\bf{Step 2:}} Let $\bar{v}_0,\ldots,\bar{v}_{k}$ be the sequence given in Lemma~\ref{seq}. Extend $V$ to $V^{+}$ as follows: Choose a brand new element $z\not\in U$ and let $V^{+}=V\cup\{(\bar{v}_0)_0^z\}$. Recall that $X$ is a non-empty set of free generators, so one can find a normal form $\varsigma\in F_0(X;n)$ such that $\tg(\bar{v}_0)\not=\varsigma$. Let $\tg((\bar{v}_0)_0^z)=\varsigma$. Define the evaluation $\ev^{+}:X\rightarrow\mathcal{P}(V^{+})$ as follows. For every $x\in X$, let $$\ev^{+}(x)=\{\bar{v}\in V^{+}:x\in \color(\tg(\bar{v}))\}.$$ 
By a similar argument to the proof of Lemma \ref{7abibi}, one can see that
\begin{equation}\label{extension}
(\forall\bar{v}\in V^{+}) \ \ \ (V^{+},\ev^{+},\bar{v})\models \tg(\bar{v}).
\end{equation}

\begin{figure}[!ht]
\centering
\begin{minipage}{\textwidth}
\centering
\begin{tikzpicture}
\draw (0,3.5) node {$V$};
\draw (0,0) ellipse (2 and 3);
\draw (0,-2.8) node {$\bullet$};
\draw (-0.3,-2.7) node {$\bar{v}_0$};
\draw (1,-2) node {$\bullet$};
\draw (1.2,-1.8) node {$\bar{v}_{1}$};
\draw (-1,-1) node {$\bullet$};
\draw (-1.3,-0.9) node {};
\draw (1,0) node {$\bullet$};
\draw (-1,1) node {$\bullet$};
\draw (1,2) node {$\bullet$};
\draw (0,2.8) node {$\bullet$};
\draw (-0.3,2.7) node {$\bar{v}_k$};
\draw (0,-2.8) -- (1,-2);
\draw (1,-2) -- (-1,-1);
\draw (-1,-1) -- (1,0);
\draw[dashed] (1,0) -- (-1,1);
\draw[dashed] (-1,1) -- (1,2);
\draw (1,2) -- (0,2.8);
\draw (5,3.5) node {$V_{+}$};
\draw (5,0) ellipse (2 and 3);
\draw (5,-2.8) node {$\bullet$};
\draw (4.7,-2.7) node {$\bar{v}_0$};
\draw (6,-2) node {$\bullet$};
\draw (6.2,-1.8) node {$\bar{v}_{1}$};
\draw (4,-1) node {$\bullet$};
\draw (3.7,-0.9) node {};
\draw (6,0) node {$\bullet$};
\draw (4,1) node {$\bullet$};
\draw (6,2) node {$\bullet$};
\draw (5,2.8) node {$\bullet$};
\draw (4.7,2.7) node {$\bar{v}_k$};
\draw (6,-3.5) node {$\bullet$};
\draw (6.5,-3.7) node {$(\bar{v}_0)_0^z$};
\draw (5,-2.8) -- (6,-2);
\draw (6,-2) -- (4,-1);
\draw (4,-1) -- (6,0);
\draw[dashed] (6,0) -- (4,1);
\draw[dashed] (4,1) -- (6,2);
\draw (6,2) -- (5,2.8);
\draw (5,-2.8) -- (6,-3.5);
\end{tikzpicture}
\end{minipage}
\end{figure}

{\bf{Step 3:}} Recall the sequence $\bar{v}_0.\ldots,\bar{v}_{k}\in V$. Let $h\in k+1$. Recall that $F_{h+1}(X;n)$ is a partition of the unit, see Theorem~\ref{andreka} \eqref{andc1}, \eqref{andc2}. Then there exists a unique normal form $\sigma_{h+1}\in F_{h+1}(X;n)$ such that $(V,\ev,\bar{v}_h)\models\sigma_{h+1}$. Similarly, there exists a unique normal form $\gamma_{h+1}\in F_{h+1}(X;n)$ such that $(V^{+},\ev^{+},\bar{v}_h)\models \gamma_{h+1}$. Thus, by \eqref{original}, \eqref{extension} and Theorem~\ref{andreka},  we have 
\begin{equation}\label{step4}
\ds_{\alpha}\models 0\not=\sigma_{h+1}\leq \tg(\bar{v}_h) \ \ \text{ and } \ \ \ds_{\alpha}\models 0\not=\gamma_{h+1}\leq \tg(\bar{v}_h).
\end{equation}

{\bf{Step 4:}} Now, we prove the following: For every $h\in k+1$, 
\begin{equation}\label{step5eq}
\ds_{\alpha}\models\sigma_{h+1}\cdot\gamma_{h+1}=0.
\end{equation}
We use induction on $h$. Since $\tg(\bar{v}_0)\not=\tg((\bar{v}_0)^z_0)=\varsigma$ and $(\forall \bar{v}\in V)\big[\bar{v}\equiv_0\bar{v}_0\implies \bar{v}=\bar{v}_0\big]$, then $\ds_{\alpha}\models\sigma_1\cdot-c_0 \varsigma\not=0$ and $\ds_{\alpha}\models\gamma_1\cdot c_0 \varsigma\not=0$. Thus, by definition of normal forms, $\ds_{\alpha}\models\sigma_1\leq -c_0 \varsigma$ and $\ds_{\alpha}\models\gamma_1\leq c_0 \varsigma$. Hence, $\ds_{\alpha}\models\sigma_1\cdot\gamma_1=0$. The induction step goes in a similar way. Suppose that $\ds_{\alpha}\models\sigma_{h+1}\cdot\gamma_{h+1}=0$, for some $h\in k$. Let $i<2$ be such that $i=h+1$ $(mod\text{ }2)$. Remember $\bar{v}_{h}\equiv_i\bar{v}_{h+1}$ and $\sigma_{h+1},\gamma_{h+1},\tg(\bar{v}_{h+1})\in F_{h+1}(X;n)$. By the induction hypothesis, without loss of generality, we may assume that $\ds_{\alpha}\models\sigma_{h+1}\cdot \tg(\bar{v}_{h+1})=0$.

Recall the construction of the unit $V$. Note that $\bar{v}_h\in\{(\bar{v}_{h+1})_i^{u}: u\in (\psi^i_{\bar{v}_{h+1}})^{*}\}$, and in fact 
\begin{equation}\label{ngbrs}
(\forall \bar{v}\in V) \ \ \ \big[\bar{v}\equiv_i\bar{v}_{h+1}\implies \bar{v}\in\{(\bar{v}_{h+1})_i^{u}: u\in (\psi^i_{\bar{v}_{h+1}})^{*}\}\cup\{\bar{v}_{h+1}\}\big]. 
\end{equation}
Remember that the labels of $(\bar{v}_{h+1})_i^{u}$'s were distinct normal forms in $\sub_i(\tg(\bar{v}_{h+1}))$. Thus, by Theorem~\ref{andreka} \eqref{andc2}, we have the following. For each element $\bar{v}\in V\setminus\{\bar{v}_{h+1},\bar{v}_{h}\}$, 
\begin{equation}\label{v}
\bar{v}\equiv_i\bar{v}_{h+1}\equiv_i\bar{v}_{h}\implies \ds_{\alpha}\models \tg(\bar{v})\cdot \tg(\bar{v}_{h})=0.
\end{equation}
 Therefore, by \eqref{step4}, \eqref{v} and the assumption that $\ds_{\alpha}\models\sigma_{h+1}\cdot \tg(\bar{v}_{h+1})=0$, we have 
\begin{equation}\label{vv}
(\forall\bar{v}\in V\setminus\{\bar{v}_{h}\}) \ \ \ \big[\bar{v}\equiv_i\bar{v}_{h+1}\implies (V,\ev,\bar{v})\not\models\sigma_{h+1} \ \text{ and } \ (V^{+},\ev^{+},\bar{v})\not\models\sigma_{h+1}\big].
\end{equation} 
Remember that $\sigma_{h+1}$ and $\gamma_{h+1}$ were chosen such that $(V,\ev,\bar{v}_{h})\models\sigma_{h+1} \ $ and $ \ (V^{+},\ev^{+},\bar{v}_{h})\models\gamma_{h+1}$. Hence, by the induction hypothesis, 
\begin{equation}\label{vh}
(V,\ev,\bar{v}_{h})\models\sigma_{h+1} \ \ \text{ and } \ \ (V^{+},\ev^{+},\bar{v}_{h})\not\models\sigma_{h+1}.
\end{equation}
We also note that $(V,\ev, \bar{v}_{h+1})\models\sigma_{h+2}$ and $(V^{+},\ev^{+}, \bar{v}_{h+1})\models\gamma_{h+2}$. Thus, by \eqref{ngbrs}, \eqref{vv} and \eqref{vh}, 
\begin{equation}
(V,\ev, \bar{v}_{h+1})\models\sigma_{h+2}\cdot c_i\sigma_{h+1} \ \ \text{ and } \ \ (V^{+},\ev^{+}, \bar{v}_{h+1})\models\gamma_{h+2}\cdot-c_i\sigma_{h+1}.
\end{equation}
Therefore, by construction of normal forms, $\ds_{\alpha}\models\sigma_{h+2}\leq c_i\sigma_{h+1}$ and $\ds_{\alpha}\models\gamma_{h+2}\leq-c_i\sigma_{h+1}$. In other words, $\ds_{\alpha}\models\sigma_{h+2}\cdot\gamma_{h+2}=0$. Hence, \eqref{step5eq} follows by the principle of mathematical induction.

In particular, there are two forms $\sigma_{k+1},\gamma_{k+1}\in F_{k+1}(X)$ each of which is satisfiable form below $\tau$ inside the free algebra $\f{X}{\alpha}$ \eqref{step4}. We also proved that these forms are disjoint \eqref{step5eq}. Therefore, $\tau$ is not an atom in the free algebra $\f{X}{\alpha}$ as desired.
\end{proof}
\subsection*{Zero dimensional elements in the free algebras} 
\begin{defn}\label{zdim}
Let $\a{A}\in\ds_{\alpha}$ and let $a\in A$. Define $\Delta a=\{i\in\alpha:c_ia\not=a\}$, the dimension set of $a$. The element $a$ is said to be zero-dimensional if and only if $\Delta a=0$.
\end{defn}
\begin{proof}[Proof of Theorem~\ref{main3}]
The free algebra $\f{\emptyset}{\alpha}$ contains only two elements $0$ and $1$. Now, suppose that $X\not=\emptyset$. Let $t\in T_{\alpha}(X)$ be such that $\ds_{\alpha}\not\models t=0$ and $\ds_{\alpha}\not\models t=1$. Then there are finite $n\in\alpha+1$ and finite $Y\subseteq X$ such that $n\geq 2$, $t\in T_{n}(Y)$ and $-t\in T_{n}(Y)$. Thus, by Theorem~\ref{andreka} (iii), one can find $k\in\omega$ and two normal forms $\tau,\sigma\in F_k(Y;n)$ such that $\ds_{\alpha}\models 0\not=\tau\leq t$ and $\ds_{\alpha}\models 0\not=\sigma\leq -t$. Now, we prove that $\ds_{\alpha}\not\models\sigma\cdot c_{i_0}\cdots c_{i_{l-1}}\tau=0$, for some $l\in \omega$ and $i_0,\ldots,i_{l-1}\in n$. Without loss of generality, we can assume that $k$ is even. 

Let $V^{\tau}$ and $V^{\sigma}$ be the two units (defined in the previous section) witnessing the satisfiability of $\tau$ and $\sigma$ respectively. We can suppose that $V^{\tau}$ and $V^{\sigma}$ share the same tail $\bar{t}$, while their actual bases are disjoint. Let $\bar{v}_0^{\tau},\ldots,\bar{v}_{k}^{\tau}\in V^{\tau}$ and $\bar{v}_0^{\sigma},\ldots,\bar{v}_{k}^{\sigma}\in V^{\sigma}$ be the sequences given by Lemma~\ref{seq}. 
Suppose that $\bar{v}_0^{\tau}=(x_0,\ldots,x_{n-1},\bar{t})$ and $\bar{v}_0^{\sigma}=(y_0,\ldots,y_{n-1},\bar{t})$. Define the following inductively, for each $j\in n$: Let $\bar{w}_1=(\bar{v}_0^{\tau})_0^{y_0}$ and let $\bar{w}_n=(\bar{w}_{n-1})_{n-1}^{y_{n-1}}$. Let $V=V^{\tau}\cup V^{\sigma}\cup\{\bar{w}_1,\ldots,\bar{w}_{n-1}\}$. 

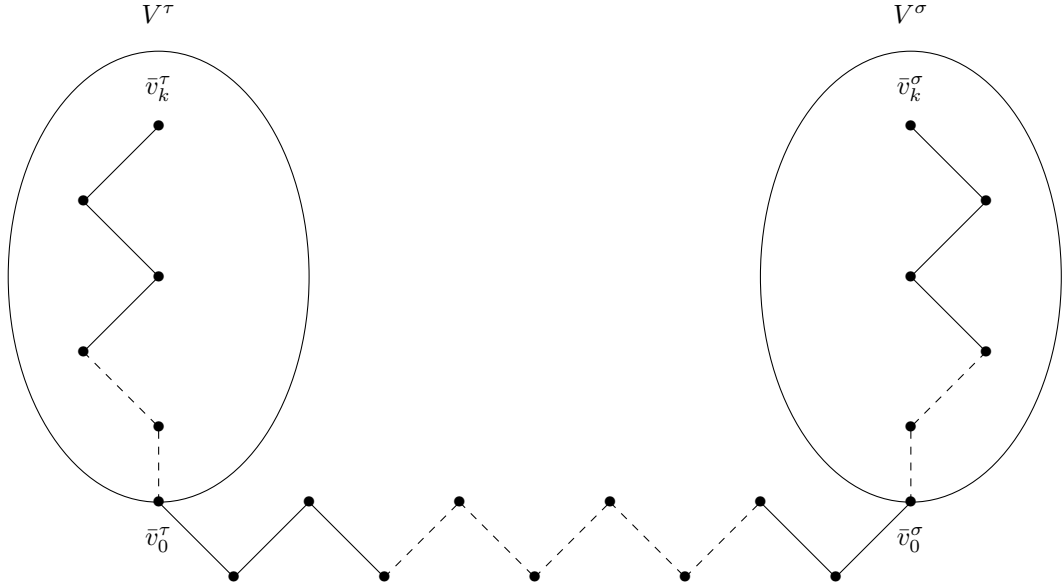
\begin{figure}[!ht]
\centering
\begin{minipage}{\textwidth}
\centering
\begin{tikzpicture}
\draw (0, 0) ellipse (2 and 3);
\draw (0,3.5) node {$V^{\tau}$};
\draw (10, 0) ellipse (2 and 3);
\draw (10,3.5) node {$V^{\sigma}$};
\draw (0,-3.5) node {$\bar{v}^{\tau}_0$};
\draw (0,-3) node {$\bullet$};
\draw (10,-3.5) node {$\bar{v}^{\sigma}_0$};
\draw (10,-3) node {$\bullet$};
\draw (1,-4) node {$\bullet$};
\draw (2,-3) node {$\bullet$};
\draw (3,-4) node {$\bullet$};
\draw (0,-3) -- (1,-4);
\draw (1,-4) -- (2,-3);
\draw (3,-4) -- (2,-3);
\draw (4,-3) node {$\bullet$};
\draw (5,-4) node {$\bullet$};
\draw (6,-3) node {$\bullet$};
\draw (7,-4) node {$\bullet$};
\draw[dashed] (3,-4) -- (4,-3);
\draw[dashed] (4,-3) -- (5,-4);
\draw[dashed] (5,-4) -- (6,-3);
\draw[dashed] (6,-3) -- (7,-4);
\draw[dashed] (7,-4) -- (8,-3);
\draw (9,-4) node {$\bullet$};
\draw (8,-3) node {$\bullet$};
\draw (10,-3) -- (9,-4);
\draw (9,-4) -- (8,-3);
\draw (0,2) node {$\bullet$};
\draw (10,2) node {$\bullet$};
\draw (0,2.5) node {$\bar{v}_k^{\tau}$};
\draw (10,2.5) node {$\bar{v}_k^{\sigma}$};
\draw (-1,1) node {$\bullet$};
\draw (0,0) node {$\bullet$};
\draw (-1,-1) node {$\bullet$};
\draw (0,-2) node {$\bullet$};
\draw (0,2) -- (-1,1);
\draw (-1,1) -- (0,0);
\draw (0,0) -- (-1,-1);
\draw[dashed] (-1,-1) -- (0,-2);
\draw[dashed] (0,-3) -- (0,-2);
\draw (11,1) node {$\bullet$};
\draw (10,0) node {$\bullet$};
\draw (11,-1) node {$\bullet$};
\draw (10,-2) node {$\bullet$};
\draw (10,2) -- (11,1);
\draw (11,1) -- (10,0);
\draw (10,0) -- (11,-1);
\draw[dashed] (11,-1) -- (10,-2);
\draw[dashed] (10,-3) -- (10,-2);
\end{tikzpicture}
\end{minipage}
\end{figure}

Recall the labels of the elements of $V^{\tau}$ and $V^{\sigma}$. For each $x\in X\setminus Y$, let $\ev(x)=\emptyset$. For each $x\in Y$, let $\ev(x)=\{\bar{v}\in V^{\tau}\cup V^{\sigma}:x\in \color(\tg(\bar{v}))\}\cup\{\bar{w}_{n-1}:n-1=1\text{ and }x\in \color(\tg(\bar{v}_0^{\tau}))\}$. By a similar argument to Lemma~\ref{7abibi}, one can verify that 
\begin{equation}
(\forall \bar{v}\in V^{\tau}\cup V^{\sigma}) \ \ \ (V,\ev,\bar{v})\models \tg(\bar{v}).
\end{equation}
Thus, $(V,\ev,\bar{v}_k^{\tau})\models \tg(\bar{v}_k^{\tau})=\tau$ 
and $(V,\ev,\bar{v}_k^{\sigma})\models \tg(\bar{v}_k^{\sigma})=\sigma$.
Moreover, it is easy to see that there are $l\in\omega$ and $i_0,\ldots,i_{l-1}\in n$ such that $(V,\ev,\bar{v}_k^{\sigma})\models\sigma\cdot c_{i_0}\cdots c_{i_{l-1}}\tau$. Hence, 
\begin{equation}\label{nozd}
\ds_{\alpha}\not\models -t\cdot c_{i_0}\cdots c_{i_{l-1}}t=0.
\end{equation}
Therefore, $t$ is not zero-dimensional in the free algebra $\f{X}{\alpha}$. Otherwise, if $t$ is zero-dimensional then $\ds_{\alpha}\models -t\cdot c_{i_0}\cdots c_{i_{l-1}}t=-t\cdot t=0$, which contradicts \eqref{nozd}. 
\end{proof}
\section{Application in logic and related developments}\label{newsection}
One way of having nice versions of first order logic is to keep the set of formulas as it is but consider generalized models when giving meaning for these formulas. Such a move was first taken by L. Henkin in \cite{hen50}. The general assignment models for first order logic, where the set of assignments of variables into a model is allowed to be an arbitrary subset of the usual one, was introduced by I. N\'emeti \cite{nem86}. With selecting a subset of assignments, dependence between variables can be introduced into semantics. For a survey on generalized semantics, see \cite{AvBBN14}.

For now, let us suppose that $\alpha\geq 2$ is finite. By a \emph{suitable language} we mean a set of $\alpha$-many individual variables together with a set of relation symbols each of which is assigned a positive rank. For simplicity, we assume that our suitable languages do not contain functional symbols and/or constant symbols. Given a suitable language $\LL$, atomic formulas are constructed in the usual way using only the relation symbols and the variables of $\LL$. There is no atom of the form $(x=y)$ unless if the identity relation appears in $\LL$ as a binary relation symbol, the identity $=$ is not treated as a privileged logical symbol in this context. The set of formulas in language $\LL$ is then defined to be the smallest set that contain all atomic formulas and which is closed under the logical connectives.

\begin{defn}
Suppose that $\LL$ is a suitable language. A general assignment model is an ordered pair $(\a{M},V)$ with $\a{M}$ a standard first order model with domain $M$ and interpretation function $I$, and $V$ is a non-empty set of assignments on $\a{M}$, i.e.,, a subset of ${^{\mathrm{VAR}} M}$, where $\mathrm{VAR}$ is the set of all individual variables of $\LL$. The language $\LL$ is interpreted as usual, now at triples $\a{M},V,s$ with $s\in V$ -  with the following clauses for quantifiers:
$$\a{M},V,s\models\exists x \varphi \myiff \ \text{ for some }t\in V: s\equiv_xt\text{ and }\a{M},V,t\models \varphi.$$
Here, $\equiv_x$ is the relation between assignments of identity up to $x$-values.
\end{defn}
We denote the logical system consists of the set of formulas in suitable language $\LL$ together with the general assignment models by $\GAM(\LL)$. The notions of \emph{satisfiable formulas}, \emph{contradictions}, \emph{valid formulas}, etc, are defined in the usual way. We note that the class of relativized diagonal-free algebras $\ds_{\alpha}$ is the algebraic counterpart of the $\GAM(\LL)$'s. Thus, the following items are the natural logical reflection of the results in Theorem~\ref{main1} and Theorem~\ref{main2}.
\begin{thm}
Let $\LL$ be a suitable language. Each of the following is true.
\begin{enumerate}
\item[(1)] $\GAM(\LL)$ is finitely-schema axiomatizable.
\item[(2)] $\GAM(\LL)$ has the finite model property, i.e.,, every non-valid formula is falsified in a finite general assignment model.
\item[(3)] The set of validities of $\GAM(\LL)$ is decidable.
\item[(4)] $\GAM(\LL)$ has most of the positive definability properties: Craig's interpolation, Beth definability, etc.
\item[(5)] If $\LL$ has at least one relation symbol, then every finitely axiomatizable theory in $\GAM(\LL)$ cannot be both complete and consistent.
\end{enumerate}
\end{thm}
Let $\a{F}(\LL)$ denotes Lindenbaum-Tarski algebra of $\GAM(\LL)$. Let $R$ be the set of all atomic formulas in language $\LL$. To deduce the above theorem from our algebraic results herein, it would be enough to prove that $\a{F}(\LL)\cong \f{R}{\alpha}$. The function assigning $\varphi^{\a{M},V}=\{s\in V : \a{M},V,s\models\varphi\}$, the meaning of $\varphi$ in $(\a{M},V)$, to $\varphi$ is a homomorphism from $\a{F}(\LL)$ to the full algebra $\a{P}(V)$. Not every homomorphism from $\a{F}(\LL)$ to $\a{P}(V)$ is of this form, though, because the meanings of the atomic formulas have to be $k$-regular in the sense that they do not distinguish sequences that agree on the first $k$ indices. Thus, $\a{F}(\LL)$ is a homomorphic image of $\f{R}{\alpha}$, but not necessarily isomorphic to it. In the literature, investigating so-called regular algebras is used to fill this gap.

We also note that a completely mechanical translation of the proofs of our algebraic theorems can be used to prove the above theorem. Such translation, from algebra to logic, was used in \cite[Chapter 2]{myphd} to obtain the following results for guarded fragments of first order logic: (1) Every satisfiable formula of guarded fragment can be extended to a finitely axiomatizable, complete and consistent theory. (2) The same is not true if we replace guarded fragment with its solo-quantifiers version; when polyadic quantifiers are not allowed. Same results hold for loosely guarded fragments, clique guarded fragments and packed fragments of first order logic. 
\subsection{Atomicity of free algebras in algebraic logic}
It was mentioned in the introduction that I. N\'emeti used a metalogical proof (translation of G\"odel's incompleteness theorem) to show non-atomicity of finitely generated free algebras of $\CA_{\alpha}$, if $\alpha\geq 3$. Such metalogical argument could be used also to deduce non-atomicity of finitely generated free algebras of other important classes of algebras of logics, e.g.,, representable cylindric algebras $\mathsf{Gs}_{\alpha}$ (if $\alpha\geq 3$), relation algebras $\mathsf{RA}$, representable relation algebras $\mathsf{RRA}$ and semi-associative relation algebras $\mathsf{SA}$. See \cite{nem85} and \cite{nem86}. 

So far, only one atomicity result has been obtained. The proof that $\a{Fr}_X\CA_{2}$ is atomic, for finite $X$, relies on the facts that $\CA_2$ is a discriminator variety and the equational theory of $\CA_2$ coincides with the equational theory of the finite $\CA_2$'s, c.f. \cite[2.5.7]{hmt1}. This could be generalized by H. Andr\'eka, B. J\'onsson and I. N\'emeti \cite{AJN91} as follows: For any variety of Boolean algebras with operators $\mathsf{K}$ of finite similarity type, if $\mathsf{K}$ is generated by its finite members then
\begin{equation*}
\mathsf{K}\text{ is a discriminator variety }\implies \a{Fr}_X\mathsf{K} \text{ is atomic, for every finite }X.
\end{equation*}

In the literature of algebraic logic, there are several varieties (of finite similarity types) that are generated by their finite members but none of them is discriminator. Here are some examples. The class of relativized cylindric set algebras $\mathsf{Crs}_{\alpha}$ and its variations $\mathsf{D}_{\alpha}$ and $\mathsf{G}_{\alpha}$. The classes of non-commutative cylindric algebras $\mathsf{WCA}_{\alpha}$ and $\mathsf{NCA}_{\alpha}$. The classes of non-associative relation algebras $\mathsf{NA}$ and weakened associative relation algebras $\mathsf{WA}$. The definitions of all these classes can be found in \cite{nem86} and \cite{nem95}. The finite algebra property of these classes can be found in \cite{nem86}, \cite{nem95}, \cite{nem87} and \cite{AHN99}. 

The question whether the finitely generated free algebras of these classes are atomic remained open for three decades. Recently, negative answers have been obtained for this problem. For finite $X$, the free algebras $\a{Fr}_X\mathsf{Crs}_{\alpha}$, $\a{Fr}_X\mathsf{D}_{\alpha}$ and $\a{Fr}_X\mathsf{G}_{\alpha}$ were shown to be not atomic in \cite{myphd}. Non-atomicity of free non-commutative cylindric algebras $\mathsf{WCA}_{\alpha}$ and $\mathsf{NCA}_{\alpha}$ was proved in \cite{myigpl}. In \cite{myadvances}, similar non-atomicity result for the class $\mathsf{NA}$ was obtained. Finally, preprint \cite{myrsl} includes an idea for showing non-atomicity of $\a{Fr}_X\mathsf{WA}$, for finite $X$. It is worthy of note that the methods in these references are quite different, in each case there is a different difficulty.

Hence, roughly speaking, we can say that being non-discriminator in the above classes was more dominant than the finite algebra property, and it caused the non-atomicity of finitely generated free algebras. A natural question arises here: is that always true?
\begin{prob}
Find a variety of Boolean algebras with operators $\mathsf{K}$ such that:
\begin{enumerate}
\item[(a)] the similarity type of $\mathsf{K}$ is finite,
\item[(b)] $\mathsf{K}$ is generated by its finite members,
\item[(b)] $\mathsf{K}$ is not discriminator variety, and
\item[(d)] all the finitely generated free algebras of $\mathsf{K}$ are atomic.
\end{enumerate}
\end{prob}
\subsection{Infinite dimensional free algebras}\label{lastsubsection} Here, assume that $\alpha\geq 2$ is an infinite ordinal. The free algebras of infinite dimensional cylindric algebras are more interesting. The metalogical technique used in \cite{nem86} provides a proof for non-atomicity of the free algebras of $\CA_{\alpha}$. The non-atomicity of free algebras of $\mathsf{Gs}_{\alpha}$ can be obtained by a purely algebraic argument, see \cite{nem84}. This argument uses the fact that $\mathsf{Gs}_{\alpha}$ is generated as a variety by its locally finite dimensional algebras. The same is not true for classes $\mathsf{Crs}_{\alpha}$, $\mathsf{D}_{\alpha}$ and $\mathsf{G}_{\alpha}$, however with a different algebraic technique non-atomicity of the free algebras of these classes was shown in \cite[Appendix 2]{myphd}.

We note that $\ds_{\alpha}$ is generated by its locally finite dimensional algebras, however the method used to prove non-atomicity of infinite dimensional free algebras of $\mathsf{Gs}_{\alpha}$ cannot work here, it depends essentially on the existence of diagonals. The method used here to prove non-atomicity of the free algebras of $\ds_{\alpha}$ is completely different than the one in \cite[Appendix 2]{myphd}.

\section*{Acknowledgment} We are deeply indebted to the anonymous referee for his fruitful comments and valuable suggestions.

\bibliographystyle{plain}

\begin{thebibliography}{100}
\bibitem{hen50}
L. Henkin (1950). 
\newblock The Completeness of Formal Systems. 
\newblock PhD thesis, Princenton University, Princeton, USA.
\bibitem{fmpd}
\newblock R. Harrop (1958).
\newblock On the existence of finite models and decision procedures for propositional calculi.
\newblock Proceedings of the Cambridge Philosophical Society, 54 (1), pp. 1-13.
\bibitem{hintikka}
J. Hintikka (1965).
\newblock Distributive Normal Forms in First-Order Logic.
\newblock In: J.N. Crossley and M.A.E. Dummett, Editor(s), Studies in Logic and the Foundations of Mathematics, Elsevier, Volume 40, pp. 48-91.
\bibitem{hmt1}
L. Henkin, J. D. Monk and A. Tarski (1971).
\newblock Cylindric Algebras Part I. 
\newblock North
Holland.
\bibitem{universal}
S. Burris and H. P. Sankappanavar (1981). 
\newblock A Course in Universal Algebra.
\newblock Graduate Texts in Mathematics, Volume 78. \newblock Springer-Verlag.
\bibitem{nem84}
I. N\'emeti (1984)
\newblock Algebraic proofs of ``$\a{Fr}_{\beta}\CA_{\alpha}$ ($\alpha\geq 4$) is not atomic''.
\newblock Manuscript, Mathematical Institute of the Hungarian Academy of Sciences, Budapest.
\bibitem{hmt2}
L. Henkin, J. D. Monk and A. Tarski (1985).
\newblock Cylindric Algebras Part II. 
\newblock North
Holland.
\bibitem{nem85}
I. N\'emeti (1985).
\newblock Logic with three variables has G{\"o}del's incompleteness property - thus free cylindric algebras are not atomic. 
\newblock Manuscript, Mathematical Institute of the Hungarian Academy of Sciences, Budapest.
\bibitem{pigozzi1}
W. J. Blok and D. Pigozzi (1986).
\newblock Protoalgebraic logics.
\newblock Studia Logica 45, pp. 337--369.
\bibitem{nem86}
I. N{\'e}meti (1986).
\newblock Free algebras and decidability in algebraic logic.
\newblock Academic Doctoral Dissertation (in Hungarian), Hungarian Academy of Sciences, Budapest.
\newblock \url{http://www.renyi.hu/~nemeti/NDis/NDis86.pdf}.
\newblock English version of chapter 3 is \cite{nem95}.
\bibitem{nem87}
I. N\'emeti (1987).
\newblock Decidability of relation algebras with weakened associativity.
\newblock Proceeding of the American Mathematical Society, 100, pp. 340--344.
\bibitem{pigozzi3}
W. J. Blok and D. Pigozzi (1989).
\newblock Algebraizable Logics.
\newblock Memoirs of the American Mathematical Society 396.
\bibitem{AMN91}
H. Andr\'eka, J. D. Monk, and I. N\'emeti, editors (1991). 
\newblock Algebraic Logic.
\newblock Colloquia Mathematica Societatis J\'anos Bolyai, vol 54. North Holland, Amsterdam.
\bibitem{AJN91}
\newblock H. Andr\'eka, B. J\'onsson and I. N\'emeti (1991).
\newblock Free algebras in discriminator varieties.
\newblock Algebra Universalis, 28 (3), pp. 401--447.
\bibitem{berghammer}
R. Berghammer, P. Kempf, G. Schmidt and T. Str\"ohlein (1991).
\newblock Relation algebras and logic of programs.
\newblock In \cite{AMN91}, pp. 37--58.
\bibitem{givant}
S. Givant (1991).
\newblock Tarski's development of Logic and Mathematics based on the calculus of relations.
\newblock In \cite{AMN91}, pp. 189--215.
\bibitem{manca-salibra}
V. Manca and A. Salibra (1991).
\newblock On the power of equational logic: applications and extensions.
\newblock \cite{AMN91}, pp. 393--412.
\bibitem{plotkin}
B. I. Plotkin (1991).
\newblock Halmos (polyadic) algebras in Database theory.
\newblock In \cite{AMN91}, pp. 503--518.
\bibitem{pigozzi2}
W. J. Blok and D. Pigozzi (1992).
\newblock Algebraic semantics for universal Horn logic without equality. 
\newblock In: A. Romanowska and J. D. H. Smith, editors, Universal Algebra and Quasigroups. Heldermann Verlag.
\bibitem{elg1}
R. Elgueta (1994).
\newblock Algebraic Model Theory for Languages without equality.
\newblock Ph.D. Thesis, Universitat de Barcelona.
\bibitem{marx}
\newblock M. Marx (1995).
\newblock Algebraic relativization and arrow logic.
\newblock Ph.D thesis, ILLC dissertation Series.
\newblock Institute for Logic, Language and Computation, University of Amsterdam.
\bibitem{nem95}
I. N{\'e}meti (1995).
\newblock Decidable versions of first order logic and cylindric-relativized set algebras.
\newblock In: Logic Colloquium'92 (Proc. Veszprem, Hungary 1992), L. Csirmaz, D. M. Gabbay and M. de Rijke, editors,
\newblock Studies in Logic, Language and Computation, CSLI Publications, pp. 177--241.
\bibitem{dell1}
E. Casanovas, P. Dellunde, R. Jansana (1996).
\newblock On elementary equivalence for Equality-free Logic.
\newblock Notre Dame Journal of Formal Logic, 37, pp. 506--522.
\bibitem{dell2}
P. Dellunde (1996).
\newblock Contributions to the Model Theory of Equality-Free Logic.
\newblock Ph.D. Thesis, Universitat de Barcelona.
\bibitem{dell3}
P. Dellunde, R. Jansana, (1996).
\newblock Some Characterization Theorems for Infinitary Universal Horn
Logic without equality.
\newblock The Journal of Symbolic Logic, 61, pp. 1242--1260.
\bibitem{elg2}
R. Elgueta (1997).
\newblock Characterizing classes defined without equality.
\newblock Studia Logica, 58, pp. 357--394.
\bibitem{andbennem}
H. Andr\'eka, J. van Benthem, and I. N\'emeti (1998). 
\newblock Modal languages and bounded fragments of predicate logic. 
\newblock Journal of Philosophical Logic, 27(3), pp. 217–274.
\bibitem{AMdNdR99}
C. Areces, C. Monz, H. de Nivelle, and M. de Rijke (1999). 
\newblock The guarded fragment: Ins and outs. 
\newblock In J. Gerbrandy, M. Marx, M. de Rijke, and Y. Venema, editors, Essays dedicated to Johan van Benthem on the occasion of his 50th birthday, volume 28, Vossiuspers, Amsterdam University Press, Amsterdam, pp. 1--14.
\bibitem{AHN99}
H. Andr\'ka, I. Hodkinson and I. N\'emeti (1999).
\newblock Finite algebras of relations are representable on finite sets.
\newblock The Journal of Symbolic Logic, 64(1), pp. 243--267.
\bibitem{monk}
J. D. Monk (2000).
\newblock An introduction to cylindric set algebras.
\newblock Logic Journal of the IGPL, 8(4), pp. 451--496.
\bibitem{HHbook}
R. Hirsch and I. Hodkinson (2002).
\newblock Relation algebras by games. 
\newblock Studies in Logic and the Foundations of Mathematics, vol 147.
\bibitem{pratt}
I. Pratt-Hartmann (2004). 
\newblock Fragments of language.
\newblock Journal of Logic, Language and Information, 13(2), pp. 207--223.
\bibitem{zalan}
Z. Gyenis (2011).
\newblock On atomicity of free algebras in certain cylindric-like varieties.
\newblock Logic Journal of the IGPL, 19 (1), pp. 44-52.
\bibitem{CA4}
H. Andr\'eka, M. Ferenczi and I. N\'emeti, editors (2013). 
\newblock Cylindric-like Algebras and Algebraic Logic. 
\newblock Bolyai Society Mathematical Studies, vol 22. Springer- Verlag.
\bibitem{andnem2013}
H. Andr\'eka and I. N\'emeti (2013). 
\newblock Reducing first-order logic logic to $\mathsf{Df}_3$, Free algebras.
\newblock In \cite{CA4}, pp. 15-35.
\bibitem{benthem}
J. van Benthem (2013).
\newblock $\mathsf{Crs}$ and guarded logic, a fruitful contact. 
\newblock In \cite{CA4}, pp. 273 -- 301.
\bibitem{duntsch} 
Ivo D\"untsch (2013).
\newblock Cylindric algebras and relational databases.
\newblock In \cite{CA4}, pp. 321--331.
\bibitem{BMP13}
P. Bourhis, M. Morak, and A. Pieris (2013). 
\newblock The impact of disjunction on query answering under guarded-based existential rules. 
\newblock In 26th International Workshop on Description Logics, DL 2013. (1014) Ulm, 2013.07.23 -2013.07.26., pp. 539--551.
\bibitem{AvBBN14}
H. Andr\'eka, J. van Benthem, N. Bezhanishvili and I. N\'emeti.
\newblock Changing a semantics: opportunism or courage?
\newblock In: M. Manzano, I. Sain, and E. Alonso, editors, The life and work of Leon Henkin. Essays on his contributions. Studies in Universal Logic, Springer, Switzerland, pp. 307--337.
\bibitem{GHKL14}
G. Gottlob, A. Hernich, C. Kupke, and T. Lukasiewicz (2014). 
\newblock Stable model semantics for guarded existential rules and description logics. 
\newblock In C. Baral and G. De Giacomo, editors, Proceedings of the 14th International Conference on the Principles of Knowledge Representation and Reasoning, AAAI. press, pp. 258--267.
\bibitem{myrsl}
M. Khaled (2015).
\newblock Weak G\"odel's incompleteness property for some decidable version of the calculus of relations.
\newblock  Preprint, arXiv:1511.01383 [math.LO].
\bibitem{myphd}
M. Khaled (2016).
\newblock G\"odel's incompleteness properties and the guarded fragment: An algebraic approach.
\newblock PhD thesis.
\newblock Central European University. \url{https://mathematics.ceu.edu/sites/mathematics.ceu.hu/files/attachment/basicpage/27/phdthesis.pdf}.
\bibitem{myigpl}
M. Khaled (2017).
\newblock The free non-commutative cylindric algebras are not atomic.
\newblock Logic journal of the IGPL, 25 (5), pp. 673--685.
\bibitem{mysh}
M. Khaled (2017).
\newblock General normal forms for any additive logic.
\newblock  Studia Scientiarum Mathematicarum Hungarica 54(3), pp. 394-409.
\bibitem{myadvances}
M. Khaled (2018).
\newblock The finitely axiomatizable complete theories of non-associative arrow frames.
\newblock Submitted.
\end{thebibliography}

\end{document}